\documentclass{article}
\usepackage{amsthm,amssymb,latexsym,graphics,amscd,amsmath,amscd,enumerate,color}
\usepackage[T1]{fontenc}
\usepackage[english]{babel}
\usepackage{multicol}
\usepackage{multirow}
\usepackage{array}
\usepackage{booktabs}
\usepackage[usenames,dvipsnames,svgnames,table]{xcolor}
\usepackage{graphicx}
\usepackage{tikz}
\usetikzlibrary{matrix}
\usepackage{wrapfig}
\usepackage{graphics}
\usepackage{lmodern} \normalfont
\usepackage{mathrsfs}
\usepackage{latexsym}
\usepackage{yfonts}
\usepackage[all,cmtip]{xy}
\usepackage{mathtools}

\usepackage{ textcomp }

\usepackage[margin=0.5in]{geometry}

\newcommand{\Ccl}{\mathbb{C}\ell}
\newcommand{\Cl}{\mathcal{C}\ell}

\newcommand{\mb}{\mathbb}

\newcommand{\Sp}{\mbox{Spin}}
\newcommand{\Spc}{\mbox{Spin}^{\mb{C}}}

\newtheorem{theorem}{Theorem}

\newtheorem{lemma}[theorem]{Lemma}

\begin{document}

%
%
%
%
%
%
%
%
%

\title{{Immersion in $\mathbb{S}^n$ by complex spinors}}

\author{Rafael de Freitas Le\~ao and Samuel Augusto Wainer
}

\maketitle


\begin{abstract}

Since the first work of Thomas Friedrich showing that isometric immersions of Riemann surfaces are related to spinors and the Dirac equation, various works appeared generalizing this approach to more general $\Sp$-manifolds, in particular the case of submanifolds of $\Sp$-manifolds of constant curvature. In the present work we investigate the case of submanifolds of $\Spc$-manifolds of constant curvature.

\end{abstract}

\maketitle
\section{Introduction}

A century long classical problem on Differential Geometry is the study of isometric immersions of riemannian manifolds. Classically, this problem is studied using generalized forms of the Gauss-Codazzi equations. But in the special case of riemann surfaces there is the approach of the Weierstrass map using complex analisys. Recently, this problem gained a new impetus when Friedrich, \cite{friedrich98}, discovered that the eierstrass map can be described using spinors.

Since than, numerous works appeared, \cite{morel04,lawn08,lawnroth10,lawnroth11,bayard13,bayard13b,bayard17}, showing how Dirac equations, spinors, Gauss-Codazzi equations and isometric immersions are related. In particular, Bayard et al, \cite{bayard17}, showed how to generalize de concept of the spinorial Weierstrass map to arbitrary dimensional spin manifolds.

In, \cite{leaowainer2018}, we argued that in certain contexts, particularly for complex manifolds, the hypothesis of a $\Sp$-structure is somewhat restrictive, being more natural to consider $\Spc$-structures, and showed how the Weierstrass map constructed by Bayard can be adapted to this case.

On \cite{bayard17}, spinor techniques are also used to investigate the more general problem of isometric immersions of manifolds on manifolds of constant curvature. As usual, to do follow the spinor approach we must assume that the manifolds involved carry a $\Sp$-structure and, again, there are some cases, like complex manifolds, where this assumption is more restrictive than the assumption of a $\Spc$-structure. In the present work we will consider Spinorial Representation of Submanifolds in $\mathbb{S}^{n}$. In particular we prove the following theorem:

\begin{theorem}
	Let $M$ $p$-dimensional manifold, $E\rightarrow M$ a vector bundle of rank $q$, assume that $TM$ and $E$ are oriented and $Spin^{\mathbb{C}}.$ Suppose that $B:TM\times TM\rightarrow E$ is symmetric and bilinear. The following are equivalent:
	
	\begin{enumerate}
		\item There exist a section $\varphi \in \Gamma (S\sum\nolimits^{\mathbb{C}})$ such that 
		\begin{equation*}
		\nabla _{X}^{\Sigma ^{\mathbb{C}}}\varphi =-\frac{1}{2}\sum_{i=1}^{p}e_{i}\cdot B(X,e_{i})\cdot \varphi +\frac{1}{2} X \cdot \nu \cdot \varphi +\frac{1}{2}\mathbf{i}A^{l}(X)\cdot
	\varphi  ,~~\forall
		X\in TM.  
		\end{equation*}
		
		\item There exist an isometric immersion $F:M\rightarrow \mathbb{S}^{n }$ with normal bundle $E$ and second fundamental form $B$.
	\end{enumerate}
	
	Furthermore, $F=\left\langle \left\langle \nu \cdot \varphi , \varphi \right\rangle \right\rangle \in \mathbb{S}^n \subset \mathbb{R}^{n+1}.$
	
\end{theorem}

\section{Recalling some concepts}

In the previous work, \cite{leaowainer2018}, we showed how to generalize de Weierstrass map obtained by Bayard et all to the case of $\Spc$-manifolds, in particular every almost complex manifold. In this work we are interested in understanding if the $\Spc$-hypothesis is also true the case of submanifolds of manifolds of constant curvature. Therefore, in this section we will recall some concepts already presented in \cite{leaowainer2018}. 

\subsection{Adapted structures}

Let $H \rightarrow M$ be a hermitian vector bundle over $M$. A $\Spc$-structure on $H$ is defined by the following double covering

\begin{displaymath}
\xymatrixcolsep{1pc}\xymatrixrowsep{1pc}\xymatrix{ & Spin_{n}^{\mathbb{C}}
		\ar[rr]^{p^{\mathbb{C}}=\lambda^{\mathbb{C}} \times l^{\mathbb{C}}} \ar@{^{(}->}[dd] & & SO_n \times S^1 \ar@{^{(}->}[dd] \\ \mathbb{Z}_{2}
		\ar[ru] \ar[rd]& & &\\ & P_{Spin_{n}^{\mathbb{C}}}(H) \ar[rr]^{\Lambda^{\mathbb{C}}}
		\ar[rd]_{\pi^{\prime}} & & P_{SO_{n}}(H) \times_M P_{S^1}(H) \ar[ld]^{\pi} \\ & & M & }
\end{displaymath}
where $\Spc$ is the group defined by

\begin{displaymath}
	Spin_n^{\mathbb{C}} = \frac{Spin_n \times S^1}{ \{ (-1,-1) \} }, 
\end{displaymath}
and $S^1=U(1) \in \mathbb{C}$ is understood as the unitary complex numbers. As usual, a $\Spc$-structure can be viewed as a lift of the transition functions of $H$, $g_{ij}$, to the group $\Spc$, $\tilde{g}_{ij}$, but now the transition functions are classes of pairs $\tilde{g}_{ij} = \left[ (h_{ij},z_{ij}) \right]$, where $h_{ij}: U_i \cap U_j \rightarrow Spin_n$ and $z_{ij}: U_i \cap U_j \rightarrow S^1=U(1)$. 

The identity on $\Spc$ is the class $\left\{ (1,1), (-1,-1) \right\}$. Because of this, neither $h_{ij}$ or $z_{ij}$ must satisfy the cocycle condition, only the class of the pair. But, $z_{ij}^2$ satisfies the cocycle condition and defines a complex line bundle $L$, associated with the $P_{S^1}$ principal bundle in the above diagram, called the determinant of the $\Spc$-structure.

The description using transition functions is useful to make clear that $\Spc$-structures are more general than $\Sp$-structures. In fact, given a $\Sp$-structure $P_{\Sp}(H) \rightarrow P_{SO}(H)$ we immediately get a $\Spc$-structure by considering $z_{ij}=1$, in other words, by considering the trivial bundle as the determinant bundle of the structure. On the other hand \cite{hitchin}, a $\Spc$-structure produces a $\Sp$-structure iff the determinant bundle has a square root, that is, the functions $z_{ij}$ satisfies the cocycle condition. 


Another way where $\Spc$-structures are natural is  when we consider an almost complex manifold $(M,g,J)$. In this case the tangent bundle can be viewed as an $U(n)$ bundle, and the natural inclusion $U(n) \xhookrightarrow{} SO(2n)$ produces a canonical $\Spc$-structure on the tangent bundle \cite{friedrich00,nicolaescu}. For this canonical structure the determinant bundle is identified with $\wedge^{0,n} M$ and the spinor bundle constructed using an irreducible complex representation of $\Cl(2n)$ is isomorphic with $\wedge^{0,*}M = \oplus_{k=0}^n \wedge^{0,k}M$. So, various structures on spinors can be described using know structures of $M$.

Unlike the usual case for $\Sp$-structures, a metric connection on $H$ is not enough to produce a connection on $P_{\Spc}(H)$, for this, we also need a connection on the determinant bundle of the structure to get a connection on $P_{SO}(H) \times P_{S^1}(H)$ and be able to lift this connection to $P_{\Spc}(H)$.

To understand the problem of immersions using the Dirac equation in the case of $\Spc$-structures, and spinors associated to this structure, we need to understand adapted $\Spc$-structures on submanifolds. The difference to the standard $\Sp$ case is that we need to keep track of the determinant bundle. Using the ideas of \cite{bar98}, we can describe the adapted structure.

Consider a $\Spc$ $n$-dimensional manifold $Q$ and a isometrically immersed $p$-dimensional $\Spc$ submanifold $M \xhookrightarrow{} Q$. Let
\begin{displaymath}
	\begin{split}
		&P_{\Spc_{n}}(Q) \xrightarrow{\Lambda^Q} P_{SO_{n}}(Q) \times P_{S^1}(Q) \\
		& \left. P_{\Spc_{n}}(Q) \right|_M \xrightarrow{\Lambda^Q}  P_{SO_{n}}(Q) \bigg|_M \times P_{S^1}(Q) \\
		&P_{\Spc_p}(M) \xrightarrow{\Lambda^M} P_{SO_p}(M) \times P_{S^1}(M)
	\end{split}
\end{displaymath}
be the corresponding $\Spc$-structures. And let the cocycles associated to this structures be, respectively, $\tilde{g}_{\alpha \beta}$, $\left. \tilde{g}_{\alpha \beta}\right|_M  $ and $\tilde{g}_{\alpha \beta}^1$. If we define the functions $\tilde{g}_{\alpha \beta}^2$ by
\begin{displaymath}
	\tilde{g}_{\alpha \beta}^1 \tilde{g}_{\alpha \beta}^2 = \left. \tilde{g}_{\alpha \beta} \right|_M
\end{displaymath}
it is easy to see, using an adapted frame, that the two sets of functions $\tilde{g}_{\alpha \beta}^1$ and $\tilde{g}_{\alpha \beta}^2$ commutes. This implies that $\tilde{g}_{\alpha \beta}^2$ satisfies the cocycle condition, because both $\tilde{g}_{\alpha \beta}$ and $\tilde{g}_{\alpha \beta}^1$ satisfies. The cocycles $\tilde{g}_{\alpha \beta}^2$ are exactly the $\Spc$-structure for the normal bundle $\nu(M)$. With this construction, if $L$, $L_1$ and $L_2$ denotes, respectively, the determinant bundle of the $\Spc$-structure of $Q$, $M$ and $\nu(M)$ we have the relation
\begin{displaymath}
	L = L_1 \otimes L_2
\end{displaymath}


Knowing that $\nu(M)$ has a natural $\Spc$-structure we can use the left regular representation of $\Ccl(n)$ on itself to construct the following $\Spc$-Clifford bundle (this bundles will act as spinor bundles)
\begin{equation*}
	\begin{split}
		\Sigma^{\mathbb{C}}Q &:= P_{Spin^{\mathbb{C}}_{n}(Q)} 
			\times_{\rho_{n}} \mathbb{C}l_{n}, \\
		\left. \Sigma^{\mathbb{C}}Q \right|_M &:= \left.  P_{Spin^{\mathbb{C}}_{n}(Q)} \right|_M
		\times_{\rho_{n}} \mathbb{C}l_{n}, \\
		\Sigma^{\mathbb{C}} M &:= P_{Spin^{\mathbb{C}}_p(M)}\times _{\rho_{p}}\mathbb{C}l_{p}, \\
		\Sigma^{\mathbb{C}} \nu(M) &:=P_{Spin^{\mathbb{C}}_q \nu(M)} 
		\times _{\rho _{q}} \mathbb{C}l_{q}.
	\end{split}
\end{equation*}

Using the isomophism $\mathbb{C}l_p \hat{\otimes} \mathbb{C}l_q \simeq \mathbb{C}l_{n}$ and standard arguments, \cite{bar98}, we get the relation
\begin{equation*}
	\Sigma^{\mathbb{C}} Q \mid_M \simeq \Sigma^{\mathbb{C}} M \hat{\otimes} \Sigma^{\mathbb{C}} \nu(M) =: \Sigma^{\mathbb{C}}.
\end{equation*}

Let $\nabla ^{\Sigma ^{\mathbb{C}}Q},\nabla ^{\Sigma ^{\mathbb{C}}M}$ and $\nabla ^{\Sigma ^{\mathbb{C}}\nu}$ be the connection on $\sum^{\mathbb{C}}Q,\sum^{\mathbb{C}}M$ and $\sum^{\mathbb{C}}\nu(M)$ respectively, induced by the  Levi-Civita connections of $P_{SO_{n}}(Q)$, $P_{SO_{p}}(M)$, and $P_{SO_{q}}(\nu)$. We denote the connection on $\sum\nolimits^{\mathbb{C}}$ by
\begin{equation*}
\nabla ^{\Sigma ^{\mathbb{C}}} = \nabla ^{\Sigma ^{\mathbb{C}}M\otimes \Sigma^{\mathbb{C}} \nu}:=\nabla
^{\Sigma ^{\mathbb{C}}M}\otimes Id+Id\otimes \nabla ^{\Sigma ^{\mathbb{C}}\nu}.
\end{equation*}

The connections on these bundle are linked by the following Gauss formula:
\begin{equation} \label{gaussformula}
\nabla_{X}^{\Sigma^{\mathbb{C}} Q} \varphi =  \nabla_{X}^{\Sigma^{\mathbb{C}} } \varphi + \frac{1}{2} \sum _{i=1}^{p} e_i \cdot B(e_i,X) \cdot \varphi ,
\end{equation}
where $B:TM \times TM \rightarrow \nu(M)$ is the second fundamental form and $\left\{ e_1 \cdots e_p \right\}$ is a local orthonormal frame of $TM$. Here ``$\cdot $'' is the Clifford multiplication on $\Sigma^{\mathbb{C}}Q$.

Note that if we have a parallel spinor $\varphi$ in $\Sigma ^{\mathbb{C}}Q$, for exemple if $Q=\mathbb{R}^{n}$, then Eq.(\ref{gaussformula}) implies the following generalized Killing equation
\begin{equation}
    \nabla_{X}^{\Sigma^{\mathbb{C}} } \varphi =- \frac{1}{2} \sum _{i=1}^{p} e_i \cdot B(e_i,X) \cdot \varphi .
\end{equation}

\subsection{A $\mathbb{C}l_{n}$-valued inner product}

To obtaining an immersion using spinors that satisfies certain equations, we need the following $\Ccl_n$-valued inner product

\begin{eqnarray*}
		\tau  :\Ccl_{n}& \rightarrow &\Ccl_{n} \\
		\tau (a~e_{i_{1}}e_{i_{2}}\cdots e_{i_{k}}) &:=&(-1)^{k}\bar{a}
			~e_{i_{k}}\cdots e_{i_{2}}e_{i_{1}}, \\
		\tau (\xi) &:=& \overline{\xi}	
	\end{eqnarray*}

\begin{equation*}
	\begin{split}
		\left\langle \left\langle \cdot ,\cdot \right\rangle \right\rangle:
			\mathbb{C}l_{n}\times \mathbb{C}l_{n} &\rightarrow \mathbb{C}l_{n} \\
		(\xi _{1},\xi _{2}) &\mapsto \left\langle \left\langle \xi _{1}, 
			\xi_{2}\right\rangle \right\rangle =\tau (\xi _{2})\xi _{1}.
	\end{split} \label{product}
\end{equation*}

\begin{equation*}
	\begin{split}
	\left\langle \left\langle (g\otimes s)\xi _{1},(g\otimes s) 
		\xi_{2}\right\rangle \right\rangle  &=s\overline{s}\tau 
		(\xi _{2})\tau (g)g \xi_{1}=\tau (\xi _{2})\xi _{1} = 
		\left\langle \left\langle \xi _{1}, \xi_{2}\right\rangle 
		\right\rangle , \\
	g\otimes s &\in Spin_{n}^{\mathbb{C}}\subset \mathbb{C}l_{n},
	\end{split}
\end{equation*}
so the product is well defined on the $\Spc$-Clifford bundles, i.e., Eq.(\ref{product}) induces a $\mathbb{C}l_{n}$-valued map:
\begin{gather*}
\sum\nolimits^{\mathbb{C}}Q\times \sum\nolimits^{\mathbb{C}}Q\rightarrow 
\mathbb{C}l_{n} \\
(\varphi _{1},\varphi _{2})=([p,[\varphi _{1}]],[p,[\varphi _{2}]])\mapsto
\left\langle \left\langle \lbrack \varphi _{1}],[\varphi _{2}]\right\rangle
\right\rangle =\tau ([\varphi _{2}])[\varphi _{1}],
\end{gather*}
where $[\varphi _{1}]$, $[\varphi _{2}]$ are the representative of $\varphi_{1},\varphi _{2}$ in the $Spin^{\mathbb{C}}_n$ frame $p\in P_{Spin^{\mathbb{C}}_n}.$

\begin{lemma}
	The connection $\nabla ^{\Sigma ^{\mathbb{C}}Q}$ is compatible with the	product $\left\langle \left\langle \cdot ,\cdot \right\rangle \right\rangle .$
\end{lemma}

\begin{proof}
	Fix $s=(e_{1},...,e_{n}):U\subset M\subset Q\rightarrow P_{SOn}$ a local section of the frame bundle, $l:U\subset M\subset Q\rightarrow P_{S^{1}}$ a local section of the associated $S^1$-principal bundle, $w^{Q}:T(P_{SO(n)})\rightarrow so(n)$ is the Levi-Civita connection of $P_{SO(n)}$ and $iA:TP_{S^{1}}\rightarrow i\mathbb{R}$ is an arbitrary connection on $P_{S^{1}}$, denote by $w^{Q}(ds(X))=(w_{ij}(X))\in so(n),$ $iA(dl(X))=iA^{l}(X).$
	
	If $\psi =[p,[\psi ]]\ $and $\psi ^{\prime }=[p,[\psi ^{\prime }]]$ are	sections of $\sum\nolimits^{\mathbb{C}}Q$ we have: 
	\begin{eqnarray*}
		\nabla _{X}^{\Sigma ^{\mathbb{C}}Q}\psi &=&\left[ p,X([\psi ])+\frac{1}{2} \sum_{i<j}w_{ij}(X)e_{i}e_{j}\cdot \lbrack \psi ]+\frac{1}{2} iA^{l}(X)[\psi ]\right] , \\
		\left\langle \left\langle \nabla _{X}^{\Sigma ^{\mathbb{C}}Q}\psi ,\psi^{\prime }\right\rangle \right\rangle &=&\overline{[\psi ^{\prime }]}\left( X([\psi ])+\frac{1}{2}\sum_{i<j}w_{ij}e_{i}e_{j}\cdot \lbrack \psi]+\frac{1}{2}iA^{l}(X)[\psi ]\right) , \\
		\left\langle \left\langle \psi ,\nabla _{X}^{\Sigma ^{\mathbb{C}}Q}\psi' \right\rangle \right\rangle &=&\overline{\left( X([\psi'])+\frac{1}{2}\sum_{i<j}w_{ij}e_{i}e_{j}[\psi ']+\frac{1}{2}iA^{l}[\psi '] \right) }[\psi ] \\
		&=&\left( X(\overline{[\psi ']})+\frac{1}{2}\sum_{i<j}w_{ij}\overline{e_{i}e_{j}[\psi ']}+\frac{1}{2}	\overline{iA^{l}}\overline{[\psi ']}\right) [\psi ] \\
		&=&\left( X(\overline{[\psi ']})-\frac{1}{2}\sum_{i<j}w_{ij}\overline{[\psi ']}e_{i}e_{j}-\frac{1}{2}iA^{l} \overline{[\psi ']}\right) [\psi ],
	\end{eqnarray*}
	then
	\begin{eqnarray*}
		\left\langle \left\langle \nabla _{X}^{\Sigma ^{\mathbb{C}}Q}\psi ,\psi^{\prime }\right\rangle \right\rangle +\left\langle \left\langle \psi,\nabla _{X}^{\Sigma ^{\mathbb{C}}Q}\psi ^{\prime }\right\rangle\right\rangle &=&\overline{[\psi ^{\prime }]}X(\xi )+X(\overline{[\psi^{\prime }]})[\psi ], \\
		X\left\langle \left\langle \psi ,\psi ^{\prime }\right\rangle \right\rangle&=&X\left( \overline{\xi ^{\prime }}\xi \right) =X(\overline{\xi ^{\prime }})\xi +\overline{\xi ^{\prime }}X(\xi ).
	\end{eqnarray*}
\end{proof}

\begin{lemma}
	The map $\left\langle \left\langle \cdot ,\cdot \right\rangle \right\rangle:\sum\nolimits^{\mathbb{C}}Q\times \sum\nolimits^{\mathbb{C}}Q\rightarrow \mathbb{C}l_{n}$ satisfies:
	
	\begin{enumerate}
		\item $\left\langle \left\langle X\cdot \psi ,\varphi \right\rangle
		\right\rangle =-\left\langle \left\langle \psi ,X\cdot \varphi \right\rangle
		\right\rangle ,~\psi ,\varphi \in \sum\nolimits^{\mathbb{C}}Q,~X\in TQ.$
		
		\item $\tau \left\langle \left\langle \psi ,\varphi \right\rangle
		\right\rangle =\left\langle \left\langle \varphi ,\psi \right\rangle
		\right\rangle ,~\psi ,\varphi \in \sum\nolimits^{\mathbb{C}}Q$
	\end{enumerate}
	
	\begin{proof}This is an easy calculation:
		\begin{enumerate} 
		
			\item $\left\langle \left\langle X\cdot \psi ,\varphi \right\rangle
			\right\rangle =\tau \lbrack \varphi ][X\cdot \psi ]=\tau \lbrack \varphi
			][X][\psi ]=-\tau \lbrack \varphi ]\tau \lbrack X][\psi ]=\left\langle
			\left\langle \psi ,X\cdot \varphi \right\rangle \right\rangle $
			
			\item $\tau \left\langle \left\langle \psi ,\varphi \right\rangle
			\right\rangle =\tau (\tau \lbrack \varphi ][\psi ])=\tau \lbrack \psi
			][\varphi ]=\left\langle \left\langle \varphi ,\psi \right\rangle
			\right\rangle .$
		\end{enumerate}
	\end{proof}
\end{lemma}
Note the same idea, product and properties are valid for the bundles $\sum^{%
	\mathbb{C}}Q,$ $\sum^{\mathbb{C}}M$, $\sum^{\mathbb{C}} \nu(M)$, $\sum\nolimits^{%
	\mathbb{C}}M\hat{\otimes}\sum\nolimits^{\mathbb{C}} \nu(M).$

\section{Spinorial Representation of Submanifolds in $\mathbb{S}^{n}$} \label{immersion}

\subsection{Adapted $Spin^{\mathbb{C}}$ groups}

Fix $n=p+q$ and consider the decomposition
$$\mathbb{R}^p \oplus \mathbb{R}^q = \mathbb{R}^n \hookrightarrow \mathbb{R}^n \oplus \mathbb{R} = \mathbb{R}^{n+1}.$$
We have natural inclusions
\begin{eqnarray*}
&& SO(p) \times SO(q) \longrightarrow SO(n) \longrightarrow SO(n+1), \\
&& Spin^{\mathbb{C}}_p \times Spin^{\mathbb{C}}_p \overset{i_1}{\longrightarrow}  Spin^{\mathbb{C}}_{n} \overset{i_2}{\longrightarrow}  Spin^{\mathbb{C}}_{(n+1)} \subset \mathbb{C}l_{(n+1)}.     
\end{eqnarray*}
Since $Spin^{\mathbb{C}}_{(n+1)}$ acts naturally on $\mathbb{C}l_{(n+1)}$ by left multiplication and by adjoint representation
\begin{eqnarray*}
    l:Spin^{\mathbb{C}}_{(n+1)} &\rightarrow& End_{\mathbb{C}}\mathbb{C}l_{(n+1)}\\
    Ad_{(n+1)}: Spin^{\mathbb{C}}_{(n+1)} &\rightarrow& End_{\mathbb{C}}\mathbb{C}l_{(n+1)} 
\end{eqnarray*}
we can define the following representations
\begin{eqnarray*}
       \rho_1 := l \circ i_2 : Spin^{\mathbb{C}}_n &\rightarrow& End_{\mathbb{C}}\mathbb{C}l_{(n+1)},\\
       \rho:=l \circ i_2 \circ  i_1  : Spin^{\mathbb{C}}_p \times Spin^{\mathbb{C}}_p &\rightarrow& End_{\mathbb{C}}\mathbb{C}l_{(n+1)},  \\
       Ad:=Ad_{(n+1)} \circ  i_2 \circ  i_1  : Spin^{\mathbb{C}}_p \times Spin^{\mathbb{C}}_p &\rightarrow& End_{\mathbb{C}}\mathbb{C}l_{(n+1)}.
\end{eqnarray*}

\subsection{Adapted Spinor Bundles}

Since $\mathbb{R}^{n+1}$ is oriented and $Spin^{\mathbb{C}}$ it induces a canonical $Spin^{\mathbb{C}}$ structure on $\mathbb{S}^n \hookrightarrow \mathbb{R}^{n+1}$:

Denote by $P_{SO(n+1)} \left( \mathbb{R}^{n+1} \right) $ the orthonormal frame bundle of $T\mathbb{R}^{n+1}$ and by $\left(  P_{SO(n+1)} \left( \mathbb{R}^{n+1} \right) \right) \big|_{\mathbb{S}^n}$ the adapted orthonormal frame bundle of the isometric immersion $\mathbb{S}^n \hookrightarrow \mathbb{R}^{n+1}$. The respective $Spin^{\mathbb{C}}$ structures are expressed by:
\begin{eqnarray*}
    &\Lambda ^{\mathbb{C}} :P_{Spin_{(n+1)}^{\mathbb{C}}}\left( \mathbb{R}^{n+1} \right) \longrightarrow
	P_{SO(n+1)}\left( \mathbb{R}^{n+1}\right) \times P_{S^{1}}\left( \mathbb{R}^{n+1} \right) , &\\
	 &\Lambda ^{\mathbb{C}}\bigg|_{\mathbb{S}^n} : \left( P_{Spin^{\mathbb{C}}_{(n+1)}}  \mathbb{R}^{n+1} \right) \bigg|_{\mathbb{S}^n}\rightarrow
    \left( P_{SO(n+1)} \mathbb{R}^{n+1} \right) \bigg|_{\mathbb{S}^n} \times \left( P_{S^{1}} \mathbb{R}^{n+1}\right) \bigg|_{\mathbb{S}^n}. &  
\end{eqnarray*}

Let $M$ a $p$-dimensional manifold, $E\rightarrow M$ a real vector bundle of rank $q$, assume that $TM$ and $E$ are oriented and $Spin^{\mathbb{C}}.$ Denote by $P_{SO(p)}(M)$ the frame bundle of $TM$ and by $P_{SO(q)}(E)$ the frame bundle of $E.$ The respective $Spin^{\mathbb{C}}$ structures are defined as
\begin{eqnarray*}
	\Lambda ^{1\mathbb{C}} &:&P_{Spin_{p}^{\mathbb{C}}}(M)\rightarrow
	P_{SO(p)}(M)\times P_{S^{1}}(M), \\
	\Lambda ^{2\mathbb{C}} &:&P_{Spin_q^{\mathbb{C}}}(E)\rightarrow
	P_{SO(q)}(E)\times P_{S^{1}}(E).
\end{eqnarray*}

Finally we are able to define the followig spinor bundles:
\begin{eqnarray*}
 	\sum\nolimits^{\mathbb{C}}\mathbb{R}^{n+1} &:=&  \left( P_{Spin^{\mathbb{C}}_{(n+1)}}\left( \mathbb{R}^{n+1} \right) \right) \times_l \mathbb{C}l_{(n+1)}, \\ 
    \sum\nolimits^{\mathbb{C}}\mathbb{S}^{n} &:=&  \left( P_{Spin^{\mathbb{C}}_{(n+1)}} \mathbb{R}^{n+1}\right) \bigg|_{\mathbb{S}^n}  \times_{\rho_1} \mathbb{C}l_{(n+1)}, \\
	\sum\nolimits^{\mathbb{C}} &:=& \left( P_{Spin^{\mathbb{C}}_p}\times
	_{M}P_{Spin^{\mathbb{C}}_q}\right) \times_\rho \mathbb{C}l_{(n+1)}, \\
	S\sum\nolimits^{\mathbb{C}} &:=& \left( P_{Spin^{\mathbb{C}}_p}\times
	_{M}P_{Spin^{\mathbb{C}}_q}\right) \times_\rho Spin_{(n+1)}^{\mathbb{C}}.
\end{eqnarray*}

In what follows we define as $\nu$ the unit vector field into adapted tangent Bundle $$T^{\mathbb{C}}:=TM \oplus E \oplus \nu := \left( P_{Spin^{\mathbb{C}}_p}\times_{M}P_{Spin^{\mathbb{C}}_q}\right) \times_{Ad} \mathbb{R}^{n+1}$$ as the one that in any spinor frame $p \in \left( P_{Spin^{\mathbb{C}}_p}\times_{M}P_{Spin^{\mathbb{C}}_q}\right) $ is written as $$\nu = [p, f_{n+1}],$$ where $f_{n+1}$ is the constant unit vector of a basis $\{f_1, \cdots f_n \} \cup \{f_{n+1}\} $ of decomposition mentioned before $\mathbb{R}^n \oplus \mathbb{R}.$

The next map will be important to us latter 
\begin{eqnarray*}
  \xi : TM \oplus E \oplus \nu \rightarrow \mathbb{C}l_{n+1}\\
  \xi (X) :=  \left\langle \left\langle X \cdot \varphi,\varphi\right\rangle \right\rangle.
\end{eqnarray*}

\subsubsection{Connection on $P_{S^1}$ Bundle}

We can define the bundle $P_{S^{1}}$ as the one with transition functions defined by product of transition functions of $P_{S^{1}}(M)$ and $P_{S^{1}}(E)$. It is not diffiult to see that there is a canonical bundle morphism: $$\Phi :P_{S^{1}}(M)\times _{M}P_{S^{1}}(E)\rightarrow P_{S^{1}}$$ such that, in any local trivialization, the following diagram comute:
\begin{center}
	\begin{equation*}
	\xymatrixcolsep{1pc}\xymatrixrowsep{1pc}\xymatrix{
		P_{S^{1}}(M)\times _{M}P_{S^{1}}(E) \ar[rr]^{~~~~~~~~~~ \Phi} \ar[dd] & &  P_{S^{1}}  \ar[dd] \\	& \\
		U_{\alpha} \times S^1 \times S^1 \ar[rr]^{~~~~~\phi_\alpha} & & U_{\alpha} \times S^1	}
	\end{equation*}
\end{center}
where $\phi_\alpha(x,r,s)=(x,rs), x\in U_{\alpha}, r,s \in S^1.$

Denote arbitrary connections on $P_{S^{1}}(M)$ and $P_{S^{1}}(E)$ by
$$\mathbf{i} A^{1}:TP_{S^{1}}(M)\rightarrow \mathbf{i} \mathbb{R}, \  \mathbf{i} A^{2}:TP_{S^{1}}(E)\rightarrow \mathbf{i} \mathbb{R}.$$ 
Express local sections by
\begin{eqnarray*}
&s=(e_{1},\cdots,e_{p}):U\rightarrow P_{SO_{p}}(M),& \\ &l_{1}:U\rightarrow P_{S^{1}}(M), \ l_{2}:U\rightarrow P_{S^{1}}(E), \ l=\Phi (l_{1},l_{2}):U\rightarrow P_{S^{1}}.&
\end{eqnarray*}

Now $iA:TP_{S^{1}}\rightarrow i\mathbb{R}$ is the connection
defined by
\begin{eqnarray*}
\mathbf{i}A(d\Phi(l_{1},l_{2}))&=&\mathbf{i}A_{1}(dl_{1})+\mathbf{i}A_{2}(dl_{2}),\\
\mathbf{i}A(dl(X))&:=&\mathbf{i}A^l(X), \ 	X\in TM.
\end{eqnarray*}

\subsection{Main Theorem}
Established the notation we have the following:
\begin{theorem}
	Let $M$ $p$-dimensional manifold, $E\rightarrow M$ a vector bundle of rank $q$, assume that $TM$ and $E$ are oriented and $Spin^{\mathbb{C}}.$ Suppose that $B:TM\times TM\rightarrow E$ is symmetric and bilinear. The following are equivalent:
	
	\begin{enumerate}
		\item There exist a section $\varphi \in \Gamma (S\sum\nolimits^{\mathbb{C}})$ such that 
		\begin{equation*}
		\nabla _{X}^{\Sigma ^{\mathbb{C}}}\varphi =-\frac{1}{2}\sum_{i=1}^{p}e_{i}\cdot B(X,e_{i})\cdot \varphi +\frac{1}{2} X \cdot \nu \cdot \varphi +\frac{1}{2}\mathbf{i}A^{l}(X)\cdot
	\varphi  ,~~\forall
		X\in TM.  \label{killing1}
		\end{equation*}
		
		\item There exist an isometric immersion $F:M\rightarrow \mathbb{S}^{n }$ with normal bundle $E$ and second fundamental form $B$.
	\end{enumerate}
	
	Furthermore, $F=\left\langle \left\langle \nu \cdot \varphi , \varphi \right\rangle \right\rangle \in \mathbb{S}^n \subset \mathbb{R}^{n+1}.$
	
\end{theorem}

\bigskip

\begin{proof}
	$2)\Rightarrow 1)$ Since $\mathbb{R}^{n+1}$ is contratible there exists a global section $$ s:\mathbb{R}^{n+1}\rightarrow P_{Spin^{\mathbb{C}}_{(n+1)}}\left( \mathbb{R}^{n+1}\right),$$
	with corresponding parallel orthonormal basis 
	\begin{eqnarray*}
	  &&h=(E_{1},\cdots ,E_{n+1}):	\mathbb{R}^{n+1}\rightarrow P_{SO(n+1)}\left( \mathbb{R}^{n+1} \right), \\  
	  && l:\mathbb{R}^{n+1}\rightarrow P_{S^{1}}\left(\mathbb{R}^{n+1}\right), \ \Lambda^{\mathbb{R}^{n+1}}(s)=(h,l).      
	\end{eqnarray*}
	Fix the constant $1=[\varphi ]\in Spin^{\mathbb{C}}_{(n+1)}\subset \mathbb{C}l_{(n+1)}$ and define the spinor field $$\varphi =[s,[\varphi ]]\in \sum\nolimits^{\mathbb{C}}
	\mathbb{R}^{n+1}:=\left(P_{Spin^{\mathbb{C}}_{(n+1)}}\left( \mathbb{R}^{n+1}\right)\right)\times_l \mathbb{C}l_{(n+1)}.$$
	Representing the connection fomrs by
	 $$w^{\mathbb{R}^{n+1}}(dh(X))=(w_{ij}^{h}(X))\in so(n+1), \  \mathbf{i} A(dl(X))=\mathbf{i} A^{l}(X)\in \mathbf{i}\mathbb{R},$$
	 we have
	\begin{eqnarray*}
	\nabla_X^{\Sigma ^{\mathbb{C}}\mathbb{R}^{n+1}}\varphi &=&\left[ s,X([\varphi ])+\left\{ 
	\frac{1}{2}\sum\nolimits_{i<j}w_{ij}^{h}(X)E_{i}E_{j}+\frac{1}{2}%
	\mathbf{i}A^{l}(X)\right\} \cdot \lbrack \varphi ]\right] \notag
	\\
	&=&\left[ s,\frac{1}{2}\mathbf{i}A^{l}(X)\cdot \lbrack \varphi ]\right]. \notag \\
    &=&\frac{1}{2}\mathbf{i}A^{l}(X)\cdot \varphi.	\label{parallel} 
	\end{eqnarray*}

If $\nu$ is the normal vector field of the immersion $\mathbb{S}^n \subset \mathbb{R}^{n+1},$ consider a local adapted orthonormal frame $$\{f_1,\cdots, f_n, \nu\}: U \rightarrow  P_{SO(n+1)}\left(\mathbb{R}^{n+1}\right) \bigg|_{\mathbb{S}^ n}.$$

Denote by $B^{\mathbb{S}^n}:T\mathbb{S}^n \times \mathbb{S}^n \rightarrow \nu (\mathbb{S}^n)$ the second fundamental form of the immersion $\mathbb{S}^n \subset \mathbb{R}^{n+1}$. 

Restricting $\varphi$ in Eq.(\ref{parallel}) to $\sum\nolimits^{\mathbb{C}}\mathbb{S}^{n}$ and applying the gauss formula Eq.(\ref{gaussformula}) we obtain
	\begin{eqnarray}
	\nabla _{X}^{\Sigma ^{\mathbb{C}}\mathbb{R}^{n+1}}\varphi -\nabla _{X}^{\Sigma ^{\mathbb{C}}\mathbb{S}^{n}}\varphi &=&\frac{1}{2}\sum_{i=1}^{n}f_{i}\cdot B^{\mathbb{S}^n}(X,f_{i})\cdot \varphi  \notag
	\\
	\frac{1}{2}\mathbf{i}A^{l}(X)\cdot \varphi -\nabla _{X}^{\Sigma ^{\mathbb{C}}\mathbb{S}^n}\varphi &=&-\frac{1}{2} X \cdot \nu \cdot \varphi  \notag
	\\
	\nabla _{X}^{\Sigma ^{\mathbb{C}}\mathbb{S}^n}\varphi &=&\frac{1}{2} X \cdot \nu \cdot \varphi +\frac{1}{2}\mathbf{i}A^{l}(X)\cdot
	\varphi .  \label{killing}
	\end{eqnarray}

Furthermore, now we can restrict $\varphi$ in Eq.(\ref{killing}) to $S\sum\nolimits^{\mathbb{C}}$ and apply again the gauss formula Eq.(\ref{gaussformula}):
	\begin{eqnarray*}
	\nabla _{X}^{\Sigma ^{\mathbb{C}}\mathbb{S}^{n}}\varphi -\nabla _{X}^{\Sigma ^{\mathbb{C}}}\varphi &=&\frac{1}{2}\sum_{i=1}^{p}e_{i}\cdot B(X,e_{i})\cdot \varphi  \notag
	\\
\frac{1}{2} X \cdot \nu \cdot \varphi +\frac{1}{2}\mathbf{i}A^{l}(X)\cdot
	\varphi -\nabla _{X}^{\Sigma ^{\mathbb{C}}}\varphi &=&\frac{1}{2}\sum_{i=1}^{p}e_{i}\cdot B(X,e_{i})\cdot \varphi  \notag
	\end{eqnarray*}

Then we prove the first part of the theorem
	
	\begin{equation*}
	\nabla _{X}^{\Sigma ^{\mathbb{C}}}\varphi =-\frac{1}{2}\sum_{i=1}^{p}e_{i}\cdot B(X,e_{i})\cdot \varphi +\frac{1}{2} X \cdot \nu \cdot \varphi +\frac{1}{2}\mathbf{i}A^{l}(X)\cdot
	\varphi .  \label{killing2}
	\end{equation*}

	\bigskip
	
	$1)\Rightarrow 2)$ The ideia here is to prove that $F=\left\langle \left\langle \nu \cdot \varphi , \varphi \right\rangle \right\rangle \in \mathbb{S}^n \subset \mathbb{R}^n$ gives us an immersion preserving the metric, the second
	fundamental form and the normal connection. For this purpose, we will
	present the following lemmas:

	\begin{lemma}
	Let $\varphi = [p,[\varphi]] \in \Gamma (S\sum\nolimits^{\mathbb{C}})$ a section satisfying Eq.(\ref{killing1}). Then:
	    \begin{enumerate}
	        \item $F:M \rightarrow \mathbb{S}^n \subset \mathbb{R}^{n+1}.$
	        \item $dF(X) =\xi(X)= \left\langle \left\langle X \cdot \varphi,\varphi\right\rangle \right\rangle, ~~\forall
		X\in TM.$ 
	\end{enumerate}
	    \begin{proof}
	        \begin{enumerate}
	            \item This follow from $$F = \left\langle \left\langle \nu \cdot \varphi,\varphi\right\rangle \right\rangle = \tau[\varphi] f_{n+1} [\varphi]= Ad(\tau[\varphi])(f_{n+1}) \in \mathbb{S}^{n} \subset \mathbb{R}^{n+1}.$$ 
	            
	            \item First note that, since $$\nu = [p,f_{n+1}] \in  \Gamma \bigg( \left( P_{Spin^{\mathbb{C}}_p}\times_{M}P_{Spin^{\mathbb{C}}_q}\right) \times_{Ad} \mathbb{R}^{n+1} \bigg) $$ we have 
	            $$\nabla^{T^{\mathbb{C}}}_X \nu = \left[p, X(f_{n+1})+ Ad_\ast \left(\omega^{\mathbb{C}}\left(dp(X) \right) \right) \left(f_{n+1} \right) \right]=0,$$
	            since $f_{n+1}$ is constant and
	            \begin{eqnarray*}
	            &&Ad(p)(f_{n+1}) = f_{n+1},~\forall p\in P_{Spin^{\mathbb{C}}_p}\times_{M}P_{Spin^{\mathbb{C}}_q}\\
	            &&\omega^{\mathbb{C}}: T \left( P_{Spin^{\mathbb{C}}_p}\times_{M}P_{Spin^{\mathbb{C}}_q}\right) \rightarrow so(n) \oplus \mathbf{i} \mathbb{R} \subset \mathbb{C}l_n. 
	            \end{eqnarray*}
	           
	            Finally:
	            \begin{eqnarray*}
	                  dF(X)&=& X(F) = X \big(\left\langle \left\langle \nu \cdot \varphi,\varphi\right\rangle \right\rangle \big)  
	                  \\
	                   &=&\left\langle \left\langle \nu \cdot \nabla^{\Sigma^{C}}\varphi,\varphi\right\rangle \right\rangle + \left\langle \left\langle \nu \cdot \varphi,\nabla^{\Sigma^{C}}_X \varphi\right\rangle \right\rangle
	                   \\
	                   &=& \left( Id - \tau \right)\left\langle \left\langle \nu \cdot \nabla_X \varphi,\varphi\right\rangle \right\rangle
	                   \\
	                   &=&-\frac{1}{2} \left( Id - \tau \right) \sum_{i=1}^{p} \left\langle \left\langle \nu \cdot e_{i}\cdot B(X,e_{i})\cdot \varphi ,\varphi\right\rangle \right\rangle 
	                   \\
	                   &&+ \frac{1}{2} \left( Id - \tau \right) \left\langle \left\langle \nu \cdot X \cdot \nu \cdot \varphi,\varphi\right\rangle \right\rangle
	                   \\
	                   &&+ \frac{1}{2} \left( Id - \tau \right) \left\langle \left\langle \mathbf{i}A^{l}(X) \nu \cdot \varphi,\varphi\right\rangle \right\rangle .
	            \end{eqnarray*}
	         
	         Using that $\nu , e_i, B(X,e_i)$ are mutually orthogonal we get 
	         \begin{eqnarray*}
	            \left\langle \left\langle \nu \cdot e_{i}\cdot B(X,e_{i})\cdot \varphi ,\varphi\right\rangle \right\rangle  &=&  \tau \bigg( \left\langle \left\langle \nu \cdot e_{i}\cdot B(X,e_{i})\cdot \varphi ,\varphi\right\rangle \right\rangle \bigg) 
	            \\
	            \left\langle \left\langle  X  \cdot \varphi,\varphi\right\rangle \right\rangle &=& -\tau \bigg( \left\langle \left\langle  X  \cdot \varphi,\varphi\right\rangle \right\rangle \bigg)
	            \\
	            \left\langle \left\langle \mathbf{i}A^{l}(X) \nu \cdot \varphi,\varphi\right\rangle \right\rangle &=& \tau \bigg(\left\langle \left\langle \mathbf{i}A^{l}(X) \nu \cdot \varphi,\varphi\right\rangle \right\rangle \bigg)
	         \end{eqnarray*}
	         Then:
	         $$dF(X) = \left\langle \left\langle X \cdot \varphi,\varphi\right\rangle \right\rangle, ~~\forall
		X\in TM. $$

	        \end{enumerate}
	    \end{proof}

	\end{lemma}

	\begin{lemma}
		\begin{enumerate} With  notations above the following statements are valid
			\item The map $F:M\rightarrow \mathbb{S}^n \subset \mathbb{R}^{n+1},$ is an isometry.
			
			\item The map%
			\begin{eqnarray*}
				\Phi _{E} &:&E\rightarrow F(M)\times \mathbb{R}^{n+1} \\
				X &\in &E_{m}\mapsto (F(m),\xi (X))
			\end{eqnarray*}%
			is an isometry between $E$ and the normal bundle of $F(M)$ into $\mathbb{S}^{ n },$ preserving connections and second fundamental forms.
		\end{enumerate}
		
		\begin{proof}
			\begin{enumerate}
				\item Let $X,Y\in \Gamma (TM\oplus E \oplus \nu \mathbb{R}),$ consequently%
				\begin{eqnarray*}
				\left\langle \xi (X),\xi (Y)\right\rangle &=&-\frac{1}{2}\left( \xi (X)\xi
				(Y)+\xi (Y)\xi (X)\right) \notag \\ &=&-\frac{1}{2}\left( \tau \lbrack \varphi
				][X][\varphi ]\tau \lbrack \varphi ][Y][\varphi ]+\tau \lbrack \varphi
				][Y][\varphi ]\tau \lbrack \varphi ][X][\varphi ]\right) \notag \\
				&=&-\frac{1}{2}\tau \lbrack \varphi ]\left( [X][Y]+[Y][X]\right) [\varphi
				]=\tau \lbrack \varphi ]\left( \left\langle X,Y\right\rangle \right)
				[\varphi ] \notag \\
				&=&\left\langle X,Y\right\rangle \tau \lbrack \varphi ][\varphi
				]=\left\langle X,Y\right\rangle .
				\end{eqnarray*}
				This implies that $F$ is an isometry, and that $\Phi _{E}$ is a bundle map
				between $E$ and the normal bundle of $F(M)$ into $\mathbb{S}^{n}$ which
				preserves the metrics of the fibers. Note that $(F(m), \xi(\nu))$ is orthogonal to $\mathbb{S}^n.$
				
				\item Denote by $B_{F}$ and $\nabla ^{\prime F}$ the second fundamental form
				and the normal connection of the immersion $F$. We want to show that:%
				\begin{eqnarray*}
					&&i)\xi (B(X,Y)) =B_{F}(\xi (X),\xi (Y)), \\
					&&ii)\xi (\nabla _{X}^{\prime }\eta ) =(\nabla _{\xi (X)}^{\prime F}\xi
					(\eta )),
				\end{eqnarray*}%
				for all $X,Y\in \Gamma (TM)$ and $\eta \in \Gamma (E)$.
				
				$i)$ First note that:
				\begin{equation*}
				B^{F}(\xi (X),\xi (Y)):=\{\nabla _{\xi (X)}^{F}\xi (Y)\}^{\bot }=\{X(\xi
				(Y))\}^{\bot },
				\end{equation*}
				where the superscript $\bot $ means that we consider the component of the
				vector which is normal to the immersion and tangent to $\mathbb{S}^n$.
				
				Supouse that in $x_{0}\in M$, $\nabla ^{M}X=\nabla ^{M}Y=0,$
				to simplify write $\nabla _{X}^{\Sigma ^{\mathbb{C}}}\varphi =\nabla
				_{X}\varphi $ and $\nabla ^{M}X=\nabla X$,
				\begin{eqnarray*}
					X(\xi (Y)) &=&\left\langle \left\langle Y\cdot \nabla _{X}\varphi ,\varphi
					\right\rangle \right\rangle +\left\langle \left\langle Y\cdot \varphi
					,\nabla _{X}\varphi \right\rangle \right\rangle \\
					 &=& (id-\tau )\left\langle
					\left\langle Y\cdot \varphi ,\nabla _{X}\varphi \right\rangle \right\rangle
					\\
				     &=&(id-\tau )\left\langle \left\langle \varphi ,\frac{1}{2}%
					\sum_{j=1}^{p}Y\cdot e_{j}\cdot B(X,e_{j})\cdot \varphi  -\frac{1}{2}Y \cdot X \cdot \nu \cdot \varphi -\frac{1}{2}\mathbf{i}A^{l}(X)Y\cdot \varphi \right\rangle \right\rangle
					\end{eqnarray*}
					\begin{eqnarray*}
			    	&=&(id-\tau )\left\langle \left\langle \varphi ,\frac{1}{2}\left(\sum_{j=1}^{p}\sum_{k=1}^{p}y^{k}e_{k}\cdot e_{j}\cdot
			    	B(X,e_{j})  -\frac{1}{2}\sum_{j=1}^{p}\sum_{k=1}^{p}y^{k}x^j e_{k}\cdot  e_{j}\cdot \nu \cdot \varphi  -\mathbf{i}A^{l}(X)Y \right) \cdot \varphi \right\rangle \right\rangle \\
			    	 &=&(id-\tau )\left\langle \left\langle \varphi ,\frac{1}{2}\left(
			    	-\sum_{j=1}^{p}y^{j}\cdot
			    	B(X,e_{j}) +\sum_{j=1}^{p}y^{j}x^j  \nu   +\sum_{j=1}^{p}\sum_{k=1,k\neq j}^{p}y^{k}e_{k}\cdot e_{j}\cdot \left( B(X,e_{j})- x^j \nu \right) -\mathbf{i}A^{l}(X)Y\right) \cdot \varphi
			    	\right\rangle \right\rangle
			    	\\
			    	  &=&(id-\tau )\left\langle \left\langle \varphi ,\frac{1}{2}\left(
			    	-B(X,Y)+D\right) \cdot \varphi \right\rangle \right\rangle 
			    	 + \left\langle Y,X \right\rangle  \left\langle \left\langle \varphi ,  \nu \cdot \varphi   \right\rangle \right\rangle,
			    	\end{eqnarray*}
				where
				\begin{eqnarray*}
					D &=& \sum_{j=1}^{p}\sum_{k=1,k\neq j}^{p}y^{k}e_{k}\cdot
					e_{j}\cdot \left( B(X,e_{j})- x^j \nu \right) -\mathbf{i}A^l(X)Y \\
					\tau \lbrack D] &=&[D].
				\end{eqnarray*}
				Consequently
				\begin{eqnarray*}
					X(\xi (Y)) &=&\frac{1}{2}(id-\tau )\left\langle \left\langle \varphi ,\left(
					-B(X,Y)+D\right) \cdot \varphi \right\rangle \right\rangle + \left\langle Y,X \right\rangle  \left\langle \left\langle \varphi ,  \nu \cdot \varphi   \right\rangle \right\rangle \\
					&=&-\tau \lbrack \varphi ]\tau \lbrack B(X,Y)][\varphi ] + \left\langle Y,X \right\rangle  \left\langle \left\langle \varphi ,  \nu \cdot \varphi   \right\rangle \right\rangle\\
					&=&\left\langle
					\left\langle \varphi ,B(X,Y)\cdot \varphi \right\rangle \right\rangle + \left\langle Y,X \right\rangle  \left\langle \left\langle \varphi ,  \nu \cdot \varphi   \right\rangle \right\rangle\\ 
					&=&\xi (B(X,Y))  + \left\langle Y,X \right\rangle  \xi(\nu).
				\end{eqnarray*}
				Therefore we conclude
				\begin{eqnarray*}
				B^{F}(\xi (X),\xi (Y))&:=&\{\nabla _{\xi
						(X)}^{F}\xi (Y)\}^{\bot }=\{X(\xi (Y))\}^{\bot } \\
					&=&\{\xi (B(X,Y))  + \left\langle Y,X \right\rangle \}^{\bot }=\xi (B(X,Y)),
				\end{eqnarray*}
				here we used the fact that $F $ is an isometry $$B(X,Y)\in
				E, \xi (B(X,Y))\in TF(M)^{\bot }, \xi(\nu)\in \{ T\mathbb{S}^n \}^{\bot}.$$ Then $i)$ follows.
				
				$ii)$ First note that
				\begin{eqnarray*}
				\nabla _{\xi (X)}^{\prime F}\xi (\eta )&=&\left\{ X(\xi (\eta ))\right\} ^{\bot
				}=\left\{ X\left\langle \left\langle \eta \cdot \varphi ,\varphi
				\right\rangle \right\rangle \right\} ^{\bot } \notag \\ &=& \left\langle \left\langle
				\nabla _{X}\eta \cdot \varphi ,\varphi \right\rangle \right\rangle ^{\bot
				}+\left\langle \left\langle \eta \cdot \nabla _{X}\varphi ,\varphi
				\right\rangle \right\rangle ^{\bot }+\left\langle \left\langle \eta \cdot
				\varphi ,\nabla _{X}\varphi \right\rangle \right\rangle ^{\bot }.
				\end{eqnarray*}%
				I will show that:%
				\begin{equation*}
				\left\langle \left\langle \eta \cdot \nabla _{X}\varphi ,\varphi
				\right\rangle \right\rangle ^{\bot }+\left\langle \left\langle \eta \cdot
				\varphi ,\nabla _{X}\varphi \right\rangle \right\rangle ^{\bot }=0.
				\end{equation*}%
				In fact
				\begin{eqnarray*}
				&&\left\langle \left\langle \eta \cdot \nabla_{X} \varphi ,\varphi \right\rangle \right\rangle +\left\langle \left\langle \eta \cdot \varphi
				,\nabla _{X}\varphi \right\rangle \right\rangle = (id-\tau )\left\langle \left\langle \eta \cdot \nabla _{X}\varphi
				,\varphi \right\rangle \right\rangle \\ 
			 &=&(-id+\tau )\left\langle \left\langle \left[ \frac{1}{2}	\sum_{j=1}^{p}\eta \cdot e_{j}\cdot B(X,e_{j})\cdot \varphi  -\frac{1}{2}\eta \cdot X \cdot \nu \cdot \varphi -\frac{1}{2}\mathbf{i}A^{l}(X)\eta \cdot \varphi \right] ,\varphi \right\rangle \right\rangle 
				\\
				 &=&(-id+\tau )\left\langle \left\langle \left[ -\frac{1}{2}
				\sum_{j=1}^{p}\sum_{s=1}^{q}\sum_{k=1}^{q}a^{s}b_{j}^{k}e_{j}\cdot f_{s}\cdot f_{k}  -\frac{1}{2} \eta \cdot X \cdot \nu  -\frac{1}{2}	\mathbf{i}A^{l}(X)\eta \right] \cdot \varphi ,\varphi \right\rangle \right\rangle 
				\\
				 &=& (-id+\tau )\left\langle \left\langle \left[ \frac{1}{2}
				\sum_{j=1}^{p}\sum_{s=1}^{q}a^{s}b_{j}^{s}e_{j}  -\frac{1}{2}\sum_{j=1}^{p}\sum_{s=1}^{q}\sum_{k=1,k\neq
					s}^{q}a^{s}b_{j}^{k}e_{j}\cdot f_{s}\cdot f_{k} -\frac{1}{2} \eta \cdot X \cdot \nu  -\frac{1}{2}\mathbf{i}A^{l}(X)\eta\right]
				\cdot \varphi ,\varphi \right\rangle \right\rangle ,
				\end{eqnarray*}
				from what
				\begin{eqnarray*}
					&&\left\langle \left\langle \eta \cdot \nabla _{X}\varphi ,\varphi \right\rangle
					\right\rangle +\left\langle \left\langle \eta \cdot \varphi ,\nabla _{X}\varphi
					\right\rangle \right\rangle \\
					&=&2 \tau \lbrack \varphi ][\frac{1}{2}\sum_{j=1}^{p}
					\sum_{s=1}^{q}a^{s}b_{j}^{s}e_{j}][\varphi ] 
					=\tau \lbrack \varphi][\sum_{j=1}^{p}\sum_{s=1}^{q}a^{s}b_{j}^{s}e_{j}][\varphi ]
					\\
					&=&\tau \lbrack
					\varphi ][\zeta][\varphi ]=:\xi (\zeta)\in TF(M) \\
					&\Rightarrow &\left\langle \left\langle \eta \cdot \nabla _{X}\varphi
					,\varphi \right\rangle \right\rangle ^{\bot }+\left\langle \left\langle \eta
					\cdot \varphi ,\nabla _{X}\varphi \right\rangle \right\rangle ^{\bot }=0.
				\end{eqnarray*}%
				In conclusion%
				\begin{equation*}
				\nabla _{\xi (X)}^{F}\xi (\eta )=\left\langle
				\left\langle \nabla _{X}\eta \cdot \varphi ,\varphi \right\rangle
				\right\rangle ^{\bot }=\left( \xi (\nabla _{X}\eta )\right)^{\bot }=\xi (\nabla
				_{X}^{\prime }\eta ).
				\end{equation*}%
				At the end $ii)$ follows.
			\end{enumerate}
		
	\end{proof}
	
\end{lemma}
	
	With these Lemmas the theorem is proved.

\end{proof}

\end{document}